\documentclass[a4paper,12pt]{amsart}
\usepackage{graphicx} 

\usepackage[T1]{fontenc}
\usepackage[utf8]{inputenc}
\usepackage[english]{babel}
\usepackage{amsmath, amssymb, amsthm, amsfonts}
\usepackage{mathrsfs}
\numberwithin{equation}{section}
\usepackage{bm}
\usepackage{tikz}
\usetikzlibrary{arrows}
\usetikzlibrary{matrix}
\usepackage[left=2.5cm, right=2.5cm, top=2cm]{geometry}
\usepackage{enumerate}
\usepackage{hyperref}
\usepackage{comment}
\date{}
\theoremstyle{plain}
\newtheorem{thm}{Theorem}

\newtheorem{rema}[thm]{Remark}
\newtheorem{prop}[thm]{Proposition}

\newtheorem{lem}[thm]{Lemma}
\newtheorem{cor}[thm]{Corollary}

\numberwithin{thm}{section}
\numberwithin{figure}{section}
\title{ Perspectivity and the perspective-Schr\"oder-Bernstein Property}
\author{Dinesh Khurana}
\author{Shubham Mittal}
\thanks{The work of the second author is supported by a UGC grant and will form a part of his Ph.D. dissertation under the supervision of the first author.}


\begin{document}
	
	\begin{abstract}
	 In this paper, we study various aspects of perspectivity in modules. We prove that the following classes of modules satisfy the perspective-Schr\"oder-Bernstein property: modules with transitive perspectivity, quasi-continuous modules, Harada (and hence discrete) modules and quasi-discrete modules with the (finite) exchange property. Furthermore, we prove that for a semiregular ring, the Schr\"oder-Bernstein property implies the perspective-Schr\"oder-Bernstein property. We prove that for $A,B \subseteq ^{\oplus} M$, if all complements of $A$ are perspective with $B$, then all complements of $B$ are perspective with $A$.
     We also provide new characterizations of weakly perspective modules and modules in which perspectivity is transitive. Some applications of these results are given. We finally prove that perspectivity is transitive in a ring $R$ if and only if every special clean element of $R$ is perspective.
\end{abstract}
\maketitle
\section{Introduction}
\vspace{0.25cm}

The classical Schr\"oder-Bernstein theorem asserts that if two sets $A$ and $B$ are each bijective with a subset of the other, then $A$ and $B$ are bijective. Inspired by this, several mathematicians have studied the conditions under which two modules that are isomorphic to submodules of each other, are isomorphic. For instance, Bumby \cite{B} proved that this is true if both modules are injective.  M\"uller and Rizvi \cite[Example 1, Page 207]{MR} gave an example showing that the above result fails for quasi-continuous modules.

\medskip
On the other hand, Kaplansky \cite[Page 12]{K} raised the following question, known as the Kaplansky's First Test Problem.

\medskip
\textit{If two abelian groups are isomorphic to summands of each other, then are they isomorphic?}

\medskip
Motivated by this, the notion of the Schr\"oder-Bernstein property for modules was introduced in \cite{DER}. A module $M$ is said to satisfy the \textit{Schr\"oder-Bernstein (SB, for short) property} if for any two summands $A$ and $B$ of $M$ such that $A$ is isomorphic to a summand of $B$ and $B$ is isomorphic to a summand of $A$, then $A$ is isomorphic to $B$. In the same paper \cite[Example 2.3]{DER}, the authors gave an example of a von Neumann regular ring $R$, such that $R_R$ does not have the Schr\"oder-Bernstein property.

\medskip 
Two summands are called \textit{perspective} if they have a common complement. It is clear that two perspective submodules are isomorphic, but the converse may not be true. Recently, Keef and Ko\c{s}an in \cite{KK} defined a module $M$ to satisfy the \textit{perspective-Schr\"oder-Bernstein (PSB, for short) property} if any two summands $A$ and $B$ of $M$, such that $A$ is perspective to a summand of $B$ and $B$ is perspective to a summand of $A$, are perspective. In \cite[Theorem 6(3)]{DER}, it was proved that for any ring $R$, $R_R$ satisfies PSB iff $_RR$ satisfies PSB. In \cite{KK}, it was also proved that directly finite modules and semisimple modules satisfy the PSB-property. Extending these results in Section 2, we prove that the PSB-property holds for modules with transitive perspectivity, quasi-continuous modules, Harada modules (in particular, discrete modules) and quasi-discrete modules with the (finite) exchange property.

\medskip
In \cite{KK}, the authors also asked whether the class of modules satisfying the SB-property is distinct from the class of modules satisfying the PSB-property. We prove that in a nonsingular $(C_3)$ module, SB-property implies PSB-property. It is further shown that in a semiregular ring $R$, SB-property implies PSB-property in $R_R$.  
 
\medskip
 In \cite{M}, Mary defined an element $a \in R$ to be \textit{right perspective} if it is regular and every complement of $r_R(a)$ is perspective with $aR$. One of the main results of that paper was the left-right symmetry of this notion. In Section 3, we prove that for $A,B\subseteq^{\oplus}M$, if all complements of $A$ are perspective with $B$, then all complements of $B$ are perspective with $A$. This leads to a quick proof of Mary's \cite[Theorem 3.4]{M}. 

\medskip
For two summands $A$ and $B$ of a module $M$, by $A \sim_pB$, we mean that $A$ is perspective to $B$. In \cite{GGK}, a module $M$ is said to be \textit{perspective} if any two isomorphic summands of $M$ are perspective. In \cite[Theorem 3.4]{GGK}, several characterizations of a perspective module are given. Generalizing perspective module, Amini et al. \cite{AAM} defined a module $M$ to be \textit{weakly perspective} if any two isomorphic summands with isomorphic complements have a common complement. On the other hand, \textit{Perspectivity is transitive} in $M$ if for summands $A, B$ and $C$ of $M$, $A\sim_p B \sim_p C$ implies $A\sim_p C$. We give several characterizations of weakly perspective modules and modules in which perspectivity is transitive on the similar lines as that of \cite[Theorem 3.4]{GGK}. As an application, we give a short proof of the result \cite[Theorem 2.5]{KP1} that if perspectivity is transitive in $\,{\mathbb M}_2\bigl(R\bigr)\,$, then $R$ has stable range one. For a brief history of this problem, we refer the reader to \cite{KP1}. 

\medskip
An element $a \in R$ is said to be \textit{special clean} if $a = e +u$ for some idempotent $e$ and unit $u$ in $R$ such that $aR \cap eR=0$. In \cite[Theorem 4.1]{KP2}, it was proved that $a \in R$ is special clean if and only if some complement of $aR$ is perspective with $r_R(a)$ (which is equivalent to saying that some complement of $r_R(a)$ is perspective with $aR$). For more characterizations of a special clean element, see \cite[Lemma 2.2]{M}.  We end this paper by proving that perspectivity is transitive in a ring $R$ iff special clean elements of $R$ are perspective.

\medskip
Throughout this paper, all rings are associative with unity and modules over it are unital. The group of units and Jacobson radical are denoted by ${\rm U}(R)$ and ${\rm J}(R)$, respectively. We use the notation $A \subseteq ^{\oplus} M$ to indicate that $A$ is a summand of $M$ and $A \leq^{e} M$ denotes that $A$ is an essential submodule of a module $M$. For an element $a \in R$, $r_R(a)$ (respectively, $l_R(a)$) will denote the right (respectively, left) annihilator of $a$ in $R$. For further notations and terminology, we will follow \cite{MM}.

\vspace{0.25cm}
\section{perspective-Schr\"oder-Bernstein property}
\vspace{0.25cm}
Recall that a module $M$ is said to have \textit{perspective-Schr\"oder-Bernstein (PSB, for short) property} if for any two summands $A$ and $B$ of $M$, $A \sim_p B'\subseteq^{\oplus}B$ and $B\sim_p A'\subseteq^{\oplus}A$ implies $A \sim_p B$. A ring $R$ has PSB-property when $R_R$ (equivalently, $_RR$) does. 

\medskip
\begin{thm}
     A module in which perspectivity is transitive has the PSB-property.
\end{thm}

\noindent{\bf Proof.} Let M be a module in which perspectivity is transitive and let A and B be two summands of M such that $A \sim_p B' \subseteq^{\oplus} B$ and $B \sim_p A' \subseteq^{\oplus} A$. Then $M=A \oplus C=B' \oplus C=B \oplus D= A' \oplus D$, for some summands $C$ and $D$ of $M$. This implies that $A= A' \oplus (A\cap D)$ and $B= B' \oplus (B\cap C)$. Since $A\cap D \subseteq ^{\oplus} D$, write $D=(A\cap D) \oplus D'$ for some summand $D'$ of $D$. Then $$M=A' \oplus D=A' \oplus (A \cap D)\oplus D'=A\oplus D' \quad \text{and} \quad M=B\oplus D=B'\oplus (B\cap C)\oplus D.$$
So $D' \sim_p C$ and $C \sim_p (B \cap C)\oplus D$. By transitivity, $D' \sim_p (B \cap C)\oplus D= (B \cap C)\oplus (A \cap D)\oplus D'$. This implies $B \cap C = (0)= A \cap D$. Therefore, $A=A'$ and $B=B'$.  \qed

 
\bigskip 
We will frequently use the following known result.

\medskip 
\begin{lem} \textup{(\cite[Page 2]{H} and \cite[Lemma 4.2]{KP2})}
For submodules $A$ and $B$ of a module $M$, if $A\cong B$ and $A \cap B=(0)$, then $A \sim_p B$ in $A \oplus B$. Further, if $A\oplus B \subseteq ^{\oplus} M$, then $A\sim_p B$ in $M$.
\end{lem}

\medskip
Consider the following properties for an $R$-module $M$: 

\vspace{0.15cm}
\noindent $(C_1)$ Every submodule $N$ of $M$ is essential in a summand of $M.$

\noindent $(C_2)$ If a submodule $N$ is isomorphic to a summand of $M$, then $N$ is a summand of $M.$

\noindent $(C_3)$ If $M_1$ and $M_2$ are summands of $M$ such that $M_1 \cap M_2 =(0)$, then $M_1 \oplus M_2$ is a summand of $M$. 

\medskip
A module $M$ is called \textit{continuous} (respectively, \textit{quasi-continuous}) if it satisfies $(C_1)$ and $(C_2)$ (respectively, $(C_1)$ and $(C_3)$) and is called \textit{extending} (or $CS$) if it has $(C_1)$. It is well known that every injective module is continuous and that a continuous module is quasi-continuous.

\medskip
\begin{prop}
     Let $M$ be a nonsingular\footnote{A module $M_R$ is said to be \textit{nonsingular} if no nonzero element of $M$ is annihilated by an essential right ideal of $R$.} module with $(C_3)$ property. If $M$ satisfies the SB-property, then it also satisfies the PSB-property. In particular, if a regular ring has the SB-property, then it also has the PSB-property.
 \end{prop}

\noindent {\bf Proof.} Let A and B be two summands of M such that $A \sim_p B' \subseteq^{\oplus} B$ and $B \sim_p A' \subseteq^{\oplus} A$. Then $M=A \oplus C=B' \oplus C=B \oplus D= A' \oplus D$, for some summands $C$ and $D$ of $M$. Also $A= A' \oplus (A\cap D)$, $B= B' \oplus (B\cap C)$, $D=(A\cap D) \oplus D'$ and $C=(B\cap C) \oplus C'$ for some submodules $C'$ and $D'$. So $M= A \oplus D'=B \oplus C'$ also. As $M$ is nonsingular, $A\cap B$ has a unique essential closure in both $A$ and $B$ and hence in $M$ (see \cite[Proposition 7.44]{L}). Denote this unique closure by $E$ and write $A=E \oplus A''$, $B=E\oplus B''$. Thus, $$ M= E\oplus A''\oplus D'=E\oplus B'' \oplus (A \cap D) \oplus D'= E\oplus A''\oplus (B \cap C) \oplus C'=E\oplus B'' \oplus C'.$$ Therefore,
\begin{equation}
\begin{aligned}
A'' \sim_p B'' \oplus (A \cap D) \qquad \text{and} \qquad B'' \sim_p A'' \oplus (B \cap C).
\end{aligned}
\end{equation}

\noindent Using SB-property of $M$, $A'' \cong B'' \oplus (A \cap D) \cong A'' \oplus (B \cap C) \cong B''$. Clearly, $A'' \cap B'' =(0)$. By Lemma 2.2, $A''\sim_p B''$ in $A'' \oplus B''$ and by $(C_3)$, $A=E\oplus A'' \sim_p E\oplus B''=B$ in $M$. The last line follows from the fact that regular rings are both left and right nonsingular and satisfy $(C_3)$.   \qed

\medskip
\begin{rema}
    The conclusion of Proposition 2.3 (i.e., SB implies PSB) also holds if we take a module $M$ such that its direct summands are closed under pairwise intersection and addition.
\end{rema}

\medskip
 Idempotents $e,f \in R$ are called \textit{right associate} (denoted by $e \sim_{r} f$) if $eR=fR$ (equivalently, $f=eu$ for some unit $u$ in $R$). \textit{Left associate} idempotents are defined symmetrically. Note that idempotents $e,f \in R$ are right associate if and only if $(1-e)$ and $(1-f)$ are left associate. The notion of perspective summands in terms of idempotent chains was first studied by Diesl et al. in \cite{DDP}. It was shown that for any two idempotents $e,f \in R$, $eR \sim_p fR$ in $R_R$ if and only if there exist idempotents $g,h \in R$ such that $e\sim_r g \sim_l h \sim_r f$ (written as $e \sim_{rlr}f$). 

\medskip
The following result is a consequence of \cite[Proposition 2.4]{CP}.

\medskip
\begin{lem}
    If $e$ and $f$ are idempotents in a ring $R$ such that $\bar{e}\sim_{rlr} \bar{f}$ in $R/I$, where $I \subseteq {\rm J}(R)$ is an ideal of $R$, then $e \sim_{rlr} f$. 
\end{lem}

\noindent {\bf Proof.} Let $\bar{e}\sim_r \bar{g_1} \sim_l \bar{g_2} \sim_r \bar{f}$ be a chain of idempotents in $\bar{R}=R/I$. Then $\bar{g_2}\bar{R}=\bar{f}\bar{R}$ implies that $\bar{g_2}=\overline{f+fr(1-f)}$ for some $r\in R$. We may take $g_2=f+fr(1-f)$. Similarly, $g_1=e+es(1-e)$ for some $s\in R$. As $\bar{g_1} \sim_l \bar{g_2}$, by \cite[Proposition 2.4]{CP}, there exists an idempotent $g\in R$ such that $g_1 \sim_r g \sim_l g_2$. Therefore, $e \sim_r g_1 \sim_r g \sim_l g_2 \sim_r f$ and thus $e \sim_r g \sim_l g_2 \sim_r f$, as desired.  \qed

\bigskip
Recall that a ring $R$ is called \textit{semiregular} if $R/{\rm J}(R)$ is regular and idempotents lift modulo ${\rm J}(R)$. The following result strengthens the last part of Proposition 2.3.
\medskip 
\begin{thm}
  If a semiregular ring $S$ satisfies the SB-property, then both $\bar{S}=S/\rm{J}(S)$ and $S$ satisfy the PSB-property.
\end{thm}

\noindent {\bf Proof.} We first prove that $\bar{S}$ has the PSB-property. As $\bar{S}$ is regular, from $(2.1)$ we have $\bar{A''}\sim_p \bar{B''}\oplus (\bar{A} \cap \bar{D})$ and $\bar{B''}\sim_p \bar{A''}\oplus (\bar{B} \cap \bar{C})$. So, there exist idempotents $\bar{e}, \bar{f}, \bar{e'}, \bar{f'}, \bar{e''}$ and $\bar{f''}$ in $\bar{S}$ such that $$\bar{e}\bar{S}=\bar{A''}\oplus (\bar{B} \cap \bar{C}), \bar{f}\bar{S}= \bar{B''}\oplus (\bar{A} \cap \bar{D}), \bar{e'}\bar{S}=\bar{A''}, \bar{f'}\bar{S}=\bar{B''}, \bar{e''}\bar{S_1}=\bar{B} \cap \bar{C} \; \text{and} \; \bar{f''}\bar{S}=\bar{A} \cap \bar{D}.$$  We can assume that $\{ \bar{e'}, \bar{e''} \}$ as well as $\{ \bar{f'}, \bar{f''} \} $ are orthogonal sets of idempotents. As idempotents lift modulo ${\rm J}(S)$, suppose that $e_1, f_1, e_1', f_1', e_1''$ and $f_1''$ are idempotents in $S$ which are lifts of $\bar{e}, \bar{f}, \bar{e'}, \bar{f'}, \bar{e''}$ and $\bar{f''}$, respectively, such that both $\{ e_1', e_1'' \}$ and $\{f_1', f_1'' \} $ are orthogonal sets of idempotents. Therefore, $$e_1'S \oplus e_1''S=(e_1' + e_1'')S \subseteq ^{\oplus} S \quad\text{and} \quad f_1'S \oplus f_1''S=(f_1' +f_1'')S \subseteq ^{\oplus} S.$$ For $e_2S=e_1'S \oplus e_1''S$ and $f_2S=f_1'S \oplus f_1''S$, note that $\bar{e}\bar{S}=\bar{e_2}\bar{S}$ and $\bar{f}\bar{S}=\bar{f_2}\bar{S}$. So $\bar{e_1'} \sim_{rlr} \bar{f_2}$ and $\bar{f_1'} \sim_{rlr} \bar{e_2}$. Using Lemma 2.5, we have $e_1' \sim_{rlr} f_2$ and $f_1' \sim_{rlr} e_2$. In particular, $$e_1'S \cong f_1'S \oplus f_1''S \quad \text{and} \quad f_1'S \cong e_1'S \oplus e_1''S.$$ By SB-property of $S$, $e_1'S \cong f_1'S$. Thus $\bar{e_1'}\bar{S} \cong \bar{f_1'}\bar{S}$ i.e., $\bar{A''} \cong \bar{B''}$. Again using Proposition 2.3, $\bar{S}$ has the PSB-property. 

\vspace{0.15cm}
We now show that $S$ has the PSB-property. Suppose $A$ and $B$ are summands of $S$ such that $A \sim_p B' \subseteq^{\oplus} B$ and  $B \sim_p A' \subseteq^{\oplus} A$. In view of \cite[Lemma 6.3]{DDP}, there exist non-zero idempotents $e,f,e',f'$ in $S$ such that $e\sim_{rlr} f'$ and $f\sim_{rlr} e'$. In $\bar{S}$, we get $\bar{e}\sim_{rlr} \bar{f'}$, $\bar{f}\sim_{rlr} \bar{e'}$, $\bar{e'}\bar{S}\subseteq^{\oplus} \bar{e}\bar{S}$ and $\bar{f'}\bar{S}\subseteq^{\oplus} \bar{f}\bar{S}$. By PSB-property of $\bar{S}$, $\bar{e}\sim_{rlr}\bar{f}$. By Lemma 2.5, $e\sim_{rlr}f$ i.e., $A\sim_p B$ in $S$. Therefore, $S$ satisfies the PSB-property.    \qed

\bigskip 
A decomposition $M= \oplus_{i \in I} M_{i}$ is said to 
\textit{complement summands} if for every summand $N$ of $M$ there exists a subset $J \subseteq I$ such that $M= N \oplus  (\oplus_{j \in J} M_{j})$. It is known that in this case, each $M_i$ is indecomposable (\cite[Page 27]{MM}). A module $M$ is said to be a \textit{Harada module} if it has a decomposition $M=\oplus_{i \in I}M_i$ that complement summands and each End($M_i$) is a local ring. For equivalent characterizations of Harada modules, see \cite[Theorem 2.25]{MM}. For $S=$ End($M$), one such characterization is that $M$ is a Harada module iff $S$ is semiregular. In \cite[Theorem 2.20]{DER}, it is proved that such modules satisfy the SB-property. By Theorem 2.6, it also satisfies the PSB-property.

\medskip
\begin{cor}
    Harada modules (in particular, discrete modules\footnote{See \cite{MM} for the definition of discrete and quasi-discrete modules.}) satisfy the PSB-property.
\end{cor}

\medskip

In \cite[Theorem 15]{KK} it was proved that semisimple modules satisfy the PSB-property. The following result extends \cite[Theorem 15]{KK}.

\medskip
\begin{thm}
    Quasi-continuous modules satisfy the PSB-property.
 \end{thm}

\noindent {\bf Proof.} Let $M$ be a quasi-continuous module, $S=$End($M$) and $\bar{S}=S/\Delta$, where $\Delta=\{\alpha \in S \ | \  ker(\alpha) \leq ^{e} M \}$. By \cite[Corollary 3.13]{MM}, there exists a decomposition $\bar{S}=S_1 \times S_2$, where $S_1$ is a regular and right self-injective (in particular, continuous) ring and $S_2$ is reduced (in particular, perspective). As a continuous module satisfies the SB-property \cite[Theorem 3.17]{MM}, by Proposition 2.3, $S_1$ has the PSB-property and by Theorem 2.1, $S_2$ has the PSB-property. So $\bar{S}$ also satisfies the PSB-property (\cite[Theorem 6(4)]{KK}). Now by similar arguments in the last paragraph of Theorem 2.6 and by using \cite[Corollary 4.10]{KP2} in place of Lemma 2.5, $S$ has the PSB-property. As PSB is an ER-property\footnote{As defined by T. Y. Lam, a property of modules is said to be an \textit{ER}-property if a module $M_R$ has the property if and only if $S_S$ has the property, where $S=$End($M$).} \cite[Theorem 6(1)]{KK}, so $M$ also satisfies the PSB-property.   \qed

\medskip
It is known that discrete modules have the exchange property and for a quasi-discrete module, the finite exchange property implies the full exchange property. The following result generalize Corollary 2.7 as follows.

\medskip
\begin{thm}
    Let $M$ be a quasi-discrete module with (finite) exchange property. Then $M$ satisfies the PSB-property.
 \end{thm}

\noindent {\bf Proof.} Let $S=$End($M$) and $\bar{S}=S/\nabla$, where $\nabla=\{\alpha \in S \ | \  im(\alpha) << M \}$. Then $\bar{S}=\bar{S_1}\times \bar{S_2}$, where $\bar{S_1}$ is regular and $\bar{S_2}$ is reduced (in particular, perspective). By Theorem 2.1, $\bar{S_2}$ has the PSB-property. Note that $S_1=$ End($M_1$), where $M_1$ is a discrete module (see the proof of \cite[Proposition 5.7]{MM}). So $\nabla_1 = \rm{J}(S_1)$ and $S_1$ satisfy the SB-property (\cite[Theorem 2.13]{DER}). By Theorem 2.6, $\bar{S_1}$ satisfy the PSB-property and so is $\bar{S}$. Since $M$ has the finite exchange property, $S$ is an exchange ring. As $\nabla$ doesnot contain any non-zero idempotent, $\nabla \subseteq \rm{J}(S)$. Again, by the argument in the last paragraph of Theorem 2.6, $S$ satisfies the PSB-property. By \cite[Theorem 6(1)]{KK}, $M$ also satisfies the PSB-property.     \qed

\vspace{0.25cm}
\section{Some Results on Perspectivity}
\vspace{0.25cm}
 
In this section, we prove that for $A,B \subseteq ^{\oplus} M$, if all complements of $A$ are perspective with $B$, then all complements of $B$ are perspective with $A$. We also give some characterizations of weakly perspective modules and modules with transitive perspectivity. These results lead to quick proofs of some known results in the literature.

 \medskip
We will need the following crucial fact to prove Proposition 3.2.
\smallskip
\begin{lem}  \cite[Lemma 3.7]{KMP}
    Let $e,f\in R$ be idempotents such that $e \sim_{lrl} f+fe(1-f)$, then $e \sim_{lrl}f$ and $e \sim_{rlr}f.$
\end{lem}

\smallskip
\begin{prop}
    
 For $A,B \subseteq ^{\oplus} M$, if all complements of $A$ are perspective with $B$, then all complements of $B$ are perspective with $A$. 
\end{prop}

 \noindent{\bf Proof.} Since $eM \oplus fM=M$ is equivalent to $eS \oplus fS=S$ for idempotents $e,f \in S$=End($M$) \cite[Lemma 5.1]{GGK}, we may assume that $A,B \subseteq^{\oplus} S$. Let $Y$ be a complement of $B$ and $e,f$ be idempotents of $S$ such that $eS=B, (1-e)S=Y$ and $(1-f)S=A$. If we define an idempotent $g:= f-(1-f)(1-e)f$, then $(1-g)S=((1-f) + (1-f)(1-e)f)S= (1-f)S=A$. So $gS$ is a complement of $A$ and by assumption $e \sim _{rlr}g$, which is equivalent to say that $(1-e) \sim _{lrl}(1-g)=(1-f)+(1-f)(1-e)f$. By Lemma 3.1, $(1-e) \sim _{rlr}(1-f)$ and $(1-e) \sim _{lrl}(1-f)$. In particular, $Y$ is perspective with $A$.   \qed

\bigskip
Recall that an element $a \in R$ is called \textit{right perspective} if it is regular and every complement of $r_R(a)$ is perspective with $aR$.

 \medskip
\begin{cor}
    If $a\in R$ is right perspective, then it is also left perspective.
\end{cor}

\noindent {\bf Proof.} If $a$ is right perspective, then in view of Proposition 3.2, every complement of $aR$ is perspective with $r_R(a)$. Since each complement of $aR$ (respectively, $r_R(a)$) corresponds to $r_R(b)$ (respectively, $bR$) for some reflexive inverse $b$ of $a$, $Rb$ is perspective with $Ra$. Hence, every complement of $l_R(a)$ is perspective with $Ra$ and thus $a$ is left perspective also.    \qed

\medskip
In \cite[Theorem 3.4]{GGK} several characterizations of perspective modules were given. In the following two results, we give analogous characterizations of weakly perspective modules and modules with transitive perspectivity.

\begin{thm}
For a module $M_R$  with $S=End_R(M)$, the following conditions are equivalent:

\noindent $(1)$ $M$ is weakly perspective.

\noindent $(2)$ For any split epimorphism $\phi : M \to eM$, where $e^2=e \in S$ and Ker ($\phi$) $\cong (1-e)M$, there exists a splitting $\psi$ such that $M=$ Im($\psi$) $\oplus$ $(1-e)M$.

\noindent $(3)$  For any split epimorphism $\phi : M \to eM$, where $e^2=e \in S$ and Ker ($\phi$) $\cong (1-e)M$, there exists a splitting $\psi$ such that $e\psi$ is an automorphism of $eM$.

\noindent $(4)$ If $euse=e$, for some $e^2=e, s\in S$ and $u\in U(S)$, then $eute=e$ for some $t\in S$ such that $ete \in U(eSe)$.
\end{thm}

\noindent{\bf Proof.} The equivalences $(1) \iff (2) \iff (3)$ follow on similar lines as of \cite[Theorem 3.4]{GGK}. We only need to see that $(3) \iff (4)$.

\noindent $(3) \implies (4)$. For a unit $u\in S$, if $euse=e$ then $\phi=eu:M \to eM$ is a split epimorphism (with a splitting $se$). Also Ker($\phi$) $= u^{-1}((1-e)M) \cong (1-e)M$. Using $(3)$, there exists a splitting $\psi$ (= $te$, for some $t \in S$) such that $e\psi$ (= $ete \in eSe$) is an automorphism of $eM$.

\noindent $(4)\implies (3)$. Let $\phi : M \to eM$ be a split epimorphism with Ker($\phi$) $\cong (1-e)M$. Let $\psi'$ denotes a splitting of $\phi$ and $\mu: Ker(\phi)\to (1-e)M$ be an isomorphism. Then $u=\mu \oplus \phi : M = Ker(\phi) \oplus Im(\psi') \to (1-e)M \oplus eM$ is an automorphism of $M$. Note that Ker($\phi$) = Ker($eu$) and $eu=\phi$ on $Im(\psi')$, i.e., $\phi=eu$. The result in $(3)$ now follows by identifying Hom($eM,M$) as $Se$.  \qed

\medskip
\begin{thm}
    For a module $M_R$  with $S=End_R(M)$, the following conditions are equivalent:

\noindent $(1)$ $M$ has perspectivity transitive.

\noindent $(2)$ If $fM \sim_p eM$ in $M$ for idempotents $f$ and $e$ in $S$, then any complement of $fM$ is perspective to any complement of $eM$.

\noindent $(3)$ If $fM \sim_p eM$ in $M$ for idempotents $f$ and $e$ in $S$, then $(1-f)M \sim_p (1-e)M$.

\noindent $(4)$ For any split epimorphism $\phi : M \to eM$, where $e^2=e \in S$ and Ker ($\phi$) $\sim_p C$ for some complement $C$ of $eM$, there exists a splitting $\psi$ such that $M=$ Im($\psi$) $\oplus$ $(1-e)M$.

\noindent $(5)$  For any split epimorphism $\phi : M \to eM$, where $e^2=e \in S$ and Ker ($\phi$) $\sim_p C$ for some complement $C$ of $eM$, there exists a splitting $\psi$ such that $e\psi$ is an automorphism of $eM$.

\noindent $(6)$ If $euse=e$, for some $e^2=e, s\in S$, $u\in U(S)$ and $eu$ is special clean, then $eute=e$ for some $t\in S$ such that $ete \in U(eSe)$.
\end{thm}

\noindent {\bf Proof.} The equivalence $(1) \iff (2)$ follows from \cite[Lemma 9]{H}. $(2) \implies (3)$ is tautology. Also $(4) \iff (5)$ is quite clear. For $(5) \iff (6)$, if Ker ($\phi$) $\sim_p C$ for some complement $C$ of $eM$, then in particular, Ker ($\phi$) $\cong (1-e)M$. By Theorem 3.4, there exists a unit $u \in S$ such that $\phi=eu$. Also, $M=eM \oplus C=C \oplus X= X \oplus Ker (\phi)$ implies that Im($\phi$)$(=eM)$ and Ker ($\phi$) are perspective in complements, which is equivalent to say that $\phi = eu$ is special clean.  

\noindent $(3) \implies(4)$. Suppose $\phi: M \to eM$ is a split epimorphism with a splitting $\varphi$ and Ker ($\phi$)$\sim_p C$ for some complement $C$ of $eM$. So, $M=\varphi (M) \oplus Ker(\phi)$. Invoking $(3)$ twice, we get $\varphi(M) \sim_p eM$ and Ker($\phi$)$\sim_p (1-e)M$. Write $M=Ker(\phi) \oplus X= (1-e)M \oplus X$, for some submodule $X$ of $M$. Let $\pi:M \to X$ denotes the natural projection with kernel $(1-e)M$. Then $\pi \vert_{eM}:eM \to X$ is an isomorphism. Since Ker($\phi$)$\cap X=(0)$, $\phi \pi \vert_{eM}$ is an automorphism of $eM$. Then $\psi := \pi(\phi \pi \vert_{eM})^{-1}$ is the required splitting of $\phi$ such that Im($\psi$)$=X$ and therefore $M= Im(\psi) \oplus (1-e)M$. 

\noindent $(4) \implies (2)$. For any complements $A_1$ and $B_1$ of $fM$ and $eM$ respectively, there exist idempotents $f_1$ and $e_1$ in $S$ such that $fM=f_1M, A_1=(1-f_1)M, eM=e_1M$ and $B_1=(1-e_1)M$. As $fM \sim_p eM$, $M=f_1M \oplus C= e_1M \oplus C$ for some submodule $C$ of $M$. For the natural projection $\pi: M \to e_1M$ with kernel $C$, $\pi \vert_{f_1M}:f_1M \to e_1M$ is an isomorphism. Take $\phi:=\pi \vert_{f_1M}f_1 : M \to e_1M$. Clearly, $\phi$ is a split epimorphism (with a splitting $(\pi \vert_{f_1M})^{-1}$) and Ker($\phi$) $=(1-f_1)M \sim_p C$ (with a common complement $f_1M$). By $(4)$, There exists a splitting $\psi :e_1M \to M$ such that $M=Im(\psi) \oplus (1-e_1)M$. As $M=Im(\psi) \oplus Ker(\phi)= Im(\psi) \oplus (1-f_1)M$, we thus have $(1-e_1)M \sim_p (1-f_1)M$ i.e., $A_1 \sim_p B_1$.  \qed

\medskip
While it is already known (\cite[Proposition 4]{AAM} and \cite[Corollary 2.2]{KP1}) that weakly perspectivity and the transitivity of perspectivity are ER-properties, the same can also be observed from Theorems 3.4 and 3.5. As an application of Theorem 3.5, we give a quick proof of \cite[Theorem 2.5]{KP1} in the following result. 

\medskip
\begin{cor}\cite[Theorem 2.5]{KP1}
     If perspectivity is transitive in $\,{\mathbb M}_2\bigl(R\bigr)\,$, then $R$ has stable range one.
 \end{cor}

\noindent {\bf Proof.} Let $ax+by=1$ for some $a,b,x,y \in R$. Write $E$ = $\left(\begin{array}{cc} 1 & 0 \\ 0 & 0 \end{array}\right)$, $U$ = $\left(\begin{array}{cc} a & by \\ -1 & x \end{array}\right)$ and $S$ = $\left(\begin{array}{cc} x & 0 \\ 1 & 0 \end{array}\right)$. It is clear that $EUSE=E$. Also, $$EU=\left(\begin{array}{cc} a & 1-ax \\ 0 & 0 \end{array}\right)= \left(\begin{array}{cc} 1+x & -x(1+x) \\ 1 & -x \end{array}\right) + \left(\begin{array}{cc} a-1-x & 1-ax+x+x^2 \\ -1 & x \end{array}\right) $$ is a special clean decomposition of $EU$. By Theorem 3.5(6), there exists a $T=\left(\begin{array}{cc} t_1 & t_2 \\ t_3 & t_4 \end{array}\right)\in M_2(R)$ such that $t_1$ is a unit in $R$ and $EUTE=E$ . Thus, $at_1+byt_3=1$ and hence $a+b(yt_3t_1^{-1})$ is a unit in $R$.   \qed

\vspace{0.25cm}
\section{Transitivity of perspectivity and special clean elements}
\vspace{0.25cm}

The element-wise characterization of various classes of rings remains a compelling area of study. For instance, a ring $R$ is an internal cancellation (IC) ring iff every regular element is unit-regular (see \cite{KL} for a detailed study of IC rings). In \cite[Theorem 3.1]{M}, it was proved that a ring is perspective iff every regular element is perspective. Along similar lines of that proof, it can be easily observed that a ring is weakly perspective iff every unit-regular element is perspective. In this direction, we prove the following result:

\medskip
\begin{thm}
    Perspectivity is transitive in a ring $R$ iff every special clean element of $R$ is perspective.
\end{thm}

\noindent{\bf Proof 1.} Suppose that perspectivity is transitive in a ring $R$ and $a \in R$ be a special clean element. Then some complement (say, $X$) of $aR$ is perspective with $r_R(a)$. Since any complement $X'$ of $aR$ is perspective with $X$, by transitivity, $X'$ is perspective with $r_R(a)$ and hence $a$ is a perspective element. \\
Conversely, suppose that every special clean element is perspective. Let $eR \sim_p fR \sim_p gR$ for some idempotents $e, f$ and $g$ in $R$. Then $$R=eR\oplus X=fR\oplus X=fR\oplus Y=gR\oplus Y$$ for some $X,Y \subseteq^{\oplus}R_R$. Let $a:Y \to X$ be an isomorphism with kernel $gR$. Then $a$ is an element of $R$ such that $aR=X$ and $r_R(a)=gR$. So there is a complement (i.e., $fR$) of $aR$ which is perspective with $gR=r_R(a)$. Therefore, $a$ is a special clean and hence, perspective element of $R$. This implies $eR \sim_p gR$.    \qed

\bigskip
\noindent{\bf Proof 2.} Suppose that perspectivity is transitive in a ring $R$ and $a \in R$ be a special clean element. By \cite[Theorem 3.4]{KP2}, any pair of idempotents $e,f \in R$ such that $aR=eR$ and $r_R(a)=(1-f)R$, are connected by a right 4-chain i.e., $e \sim_{rlrl}f$. By transitivity, $e\sim_{rlr}f$, which means that $aR$ is perspective to any complement of $r_R(a)$ (i.e., $a$ is a perspective element). \\
Conversely, suppose that every special clean element is perspective. For transitivity of perspectivity, it suffices to show that $e\sim_{rlrl}f$ implies $e\sim_{rlr}f$ (for the case $e\sim_{lrlr}f$, we obtain $e\sim_re\sim_{lrl}g\sim_rf$ and hence $e\sim_{rlr}f$. Transitivity then follows by simultaneously reducing the chain). Let $a:fR \to eR$ be an isomorphism with kernel $(1-f)R$. Then $e\sim_{rlr}g\sim_lf$ implies that $aR=eR$ is perspective with $gR$ (which is a complement of $(1-g)R=(1-f)R=r_R(a)$). So $a$ is a special clean and hence perspective element of $R$, which further implies that $e \sim_{rlr} f$.  \qed

\bigskip
\noindent{\bf Proof 3.}  Let $a\in R$ be a special clean element. Then some complement (say, $X$) of $aR$ is perspective with $r_R(a)$. By \cite[Lemma 9]{H} (or see Theorem 3.5(2)), any complement of $X$ is perspective to any complement if $r_R(a)$. In particular, $aR$ is perspective to any complement of $r_R(a)$ i.e., $a$ is a perspective element. \\
Conversely, suppose that every special clean element is perspective and let $\phi:R \to eR$ be a split epimorphism with Ker($\phi$) $\sim_pC$ for some complement $C$ of $eR$. Then $\phi=eu$ is a special clean element of $R$ (see the proof of Theorem 3.5($(5) \iff (6)$)). By assumption, $eu$ is perspective and hence $r_R(eu)=$ Ker($\phi$) is perspective with $(1-e)R$. The conclusion of Theorem 3.5(4) now follows along the same lines as in $(3)\implies (4)$ of the same proof. Therefore, perspectivity is transitive in $R$.  \qed 

\vspace{0.25cm}

\vspace{0.25cm}
\section{Acknowledgements}
\vspace{0.25cm}
We are grateful to Xavier Mary for reading the paper and providing valuable comments that have improved the quality of the paper.

\vspace{.50cm}
\bigskip
\noindent Department of Mathematics \\
\noindent Panjab University \\
\noindent Chandigarh-160014, India

\medskip
\noindent {\tt dkhurana@pu.ac.in}

\smallskip
\noindent {\tt shbm66.sm@gmail.com}

\end{document}